\documentclass{amsproc}

\usepackage{amssymb,commath,mathtools}

\newtheorem{theorem}{Theorem}[section]
\newtheorem{lemma}[theorem]{Lemma}
\newtheorem{proposition}[theorem]{Proposition}
\newtheorem{corollary}[theorem]{Corollary}

\theoremstyle{definition}
\newtheorem{definition}[theorem]{Definition}

\theoremstyle{remark}

\numberwithin{equation}{section}

\newcommand{\C}{\mathbb{C}}
\newcommand{\R}{\mathbb{R}}

\newcommand{\N}{\mathbb{N}}
\newcommand{\T}{\mathbb{T}}
\newcommand{\Oct}{\mathbb{O}}

\newcommand{\vol}{\mathrm{vol}}

\newcommand{\End}{\mathrm{End}}

\newcommand{\U}{\mathrm{U}}
\newcommand{\SU}{\mathrm{SU}}
\newcommand{\Spe}{\mathrm{S}}

\newcommand{\SO}{\mathrm{SO}}
\newcommand{\so}{\mathfrak{so}}
\newcommand{\Symp}{\mathrm{Sp}}

\newcommand{\mA}{\mathcal{A}}
\newcommand{\mT}{\mathcal{T}}
\newcommand{\mH}{\mathcal{H}}
\newcommand{\mP}{\mathcal{P}}

\newcommand{\rN}{\overrightarrow{\N}}

\newcommand{\gge}{\mathfrak{e}}
\renewcommand{\gg}{\mathfrak{g}}
\newcommand{\gh}{\mathfrak{h}}
\newcommand{\gk}{\mathfrak{k}}
\newcommand{\gp}{\mathfrak{p}}

\begin{document}

\title[Radial Toeplitz operators on Cartan domains]{Radial Toeplitz operators on the weighted Bergman spaces of Cartan domains}

\author{Matthew Dawson}
\address{Centro de Investigaci\'on en Matem\'aticas, M\'erida, Mexico}
\email{matthew.dawson@cimat.mx}
\thanks{Research supported by SNI and Conacyt.}

\author{Ra\'ul Quiroga-Barranco}
\address{Centro de Investigaci\'on en Matem\'aticas, Guanajuato, Mexico}
\email{quiroga@cimat.mx}
\thanks{}


\begin{abstract}
	Let $D$ be an irreducible bounded symmetric domain with biholomorphism group $G$ with maximal compact subgroup $K$. For the Toeplitz operators with $K$-invariant symbols we provide explicit simultaneous diagonalization formulas on every weighted Bergman space. The expressions are given in the general case, but are also worked out explicitly for every irreducible bounded symmetric domain including the exceptional ones.
\end{abstract}

\maketitle

\section{Introduction}\label{sec:Introduction}
Irreducible bounded symmetric domains and the weighted Bergman spaces that they support constitute a fundamental object in functional analysis. These allow us to consider Toeplitz operators and irreducible unitary representations of Lie groups. In the last two decades, it has been observed that there is a natural and far reaching relationship between these operators and representations. This work continues the study of Toeplitz operators derived from such interaction.

For $D$ a bounded symmetric domain with biholomorphism group $G$, it was proved in \cite{DOQJFA} that there are plenty of closed subgroups $H$ in $G$ so that the Toeplitz operators with $H$-invariant symbols pairwise commute. This yields commutative $C^*$-algebras generated by Toeplitz operators. Ultimately, this is a consequence of multiplicity-free restriction results. Among the latter, the very first one is found in \cite{Schmid} for the maximal compact subgroup $K$ in $G$ that fixes a given point.

We consider the multiplicity-free decomposition for the restriction to $K$ of the unitary representations of $G$ on the Bergman spaces, as found in \cite{Schmid} or \cite{UpmeierBook}. This is used to achieve our main goal: an explicit simultaneous diagonalization of Toeplitz operators with $K$-invariant symbols. A general but also very explicit formula is obtained in Theorem~\ref{thm:Schmid_diagonal_Lebesgue}, where every $K$-invariant symbol is really seen as a complex-valued function defined on $[0,1)^r$, where $r$ is the rank of $D$. Correspondingly, our formulas are given by integrals computed in this set. Our computations make use of the spherical and conical polynomials for the representation of $K$ on the space of polynomials on the domain $D$. Hence, Theorem~\ref{thm:Schmid_diagonal} provides the simultaneous diagonalization in terms of these polynomials. Since there are explicit expressions for the conical polynomials (see \cite{Johnson}) this should prove to be useful in further developments.

On the other hand, we rely strongly on the Jordan structures associated to the irreducible bounded symmetric domains. For example, the preliminary computation given by Theorem~\ref{thm:RadialToepNormComputation} reduces the diagonalization of a Toeplitz operator with $K$-invariant symbol to integrals over the cone of the largest tube type domain contained in $D$. Most importantly, the Jordan structure allows us to consider primitives (a sort of minimal tripotent) thus making possible to obtain the formula in Theorem~\ref{thm:Schmid_diagonal_Lebesgue} that involves integrals over $[0,1)^r$. This integration is in fact given by taking coordinates with respect to a maximal collection of primitives (in the Jordan setup) or a basis of a maximal flat (in the Lie setup). In the last section we exhibit examples of these sets for every irreducible bounded symmetric domain, including the exceptional ones, which leads us to completely explicit formulas that provide the diagonalization of Toeplitz operators with $K$-invariant symbols.

The properties of Jordan pairs and analysis on irreducible bounded symmetric domains that we require are reviewed in Sections \ref{sec:JordanBSD} and \ref{sec:AnalysisBSD}, respectively. We obtain our main results for irreducible bounded symmetric domains in the last two sections: Section~\ref{sec:iso-K_and_radial-operators} with general expressions for an arbitrary domain and Section~\ref{sec:radial_Toeplitz} with explicit formulas for every irreducible case.

\section{Jordan structures associated to bounded symmetric domains}
\label{sec:JordanBSD}
In this section we collect some standard facts on bounded symmetric domains. The standard references are \cite{FarautKoranyi,FarautKoranyiBook,Helgason,Helgason2,LoosBSDJordan,UpmeierBook,Wallach1} where the reader will find further details and proofs.

Let $D$ be an irreducible Hermitian symmetric space of non-compact type, and denote by $G$ the connected component of the automorphism group of $D$ and by $K$ the subgroup of $G$ fixing some point. Consider the corresponding Cartan decomposition given by
\[
\gg = \gk \oplus \gp
\]
where $\gg$ and $\gk$ are the Lie algebras of $G$ and $K$, respectively. Choose $\gh$ a Cartan subalgebra of $\gk$, which is thus a Cartan subalgebra of $\gg$. For the complexified Lie algebras, we recall that a root $\alpha$ of $\gg^\C$ with respect to $\gh^\C$ is called compact if the corresponding root space $\gg_\alpha$ is contained in $\gk^\C$. Let us fix an order and let us denote by $\Phi^+$ the space of all positive noncompact roots. We consider the complex space
\[
\gp_+ = \sum_{\alpha \in \Phi^+} \gg_\alpha.
\]
Then, it is well known that $D$ admits a canonical embedding as a circled bounded domain in $\gp_+$ such that $K$ is precisely the isotropy subgroup of the origin: this is the Harish-Chandra realization of $D$. The complex space $\gp_+$ carries a natural Hermitian inner product obtained from the Killing form of $\gg^\C$ and the conjugation $\tau$ with respect to the real form $\gk \oplus i\gp$.

Every element $X \in \gp_+$ can be considered as a holomorphic vector field over $D$ whose one-parameter group of transformations is $t \mapsto \exp(tX)$ acting on $D$, and so $X$ can be seen as a holomorphic map $D \rightarrow \gp_+$. Then, the vector field $X$ can be written as
\[
X(z) = y - Q_y(z,z),
\]
for every $z \in D$, where $y = X(0)$ and $Q_y(z,w)$ is a complex bilinear form of $(z,w)$. Furthermore, for every $z,w$ the assignment $y \mapsto Q_y(z,w)$ is complex anti-linear. This yields a triple product
\begin{align*}
\gp_+ \times \gp_+ \times \gp_+ &\rightarrow \gp_+  \\
\{zyw\} &= Q_y(z,w),
\end{align*}
that defines a Jordan pair which also carries a natural involutive anti-linear map $z \mapsto z^*$. From now on, we will consider $\gp_+$ endowed with this algebraic structure.

For this setup, an element $z \in \gp_+$ is called tripotent if and only if $\{zzz\} = z$, and two tripotent elements $z,w$ are called orthogonal when $\{zwy\} = 0$ for all $y \in \gp_+$. A tripotent element is called primitive if and only if it is not the sum of two orthogonal tripotents. Let $e_1, \dots, e_r$ be a maximal collection of mutually orthogonal primitive tripotents. It follows from \cite{LoosBSDJordan} that $r$ is precisely the rank of $D$. It will be useful to consider the element
\[
e = e_1 + \dots + e_r,
\]
that is a tripotent element in $\gp_+$ as well.

For every tripotent $z$ there is a so-called Peirce decomposition
\[
\gp_+ = V_1(z) \oplus V_{\frac{1}{2}}(z) \oplus V_0(z),
\]
where
\[
V_\alpha(z) = \{ w \in \gp_+ \mid \{zzw\} = \alpha w\}.
\]
Then, the space $V_1(z)$ is a Jordan algebra with the product $x\circ y = \{xzy\}$ and for this structure $z$ is the identity.

From now on, we will denote
\[
V_j = V_1(e_1 + \dots + e_j)
\]
for $j = 1, \dots, r$, which is thus a Jordan algebra with identity $e_1 + \dots + e_j$. We observe that
\[
V_1 \subset \dots \subset V_r.
\]
It follows from \cite{FarautKoranyi} that $D_j = V_j \cap D$ is a bounded symmetric domain of tube type whose associated symmetric cone is given by
\[
\Omega_j = \{ x^2 \mid x \in V_j, \; x^* = x \}.
\]
Let us denote by $L_j$ the identity component of the isotropy subgroup of automorphisms of $\Omega_j$ that fix $e_1 + \dots + e_j$. By the results from \cite{FarautKoranyi,KoranyiVagi} there is a unique polynomial $\Delta_j(z)$ in $V_j$ that is $L_j$-invariant and that satisfies
\begin{equation}
\label{eq:Deltaj}
\Delta_j(t_1 e_1 + \dots + t_j e_j) = t_1 \dots t_j.
\end{equation}
We will consider the polynomials $\Delta_1(z), \dots, \Delta_r(z)$ as defined on $\gp_+$ by orthogonal projection.

For the case $j = r$ we drop the index, and so we have $V = V_r$, $\Omega = \Omega_r$, $\Delta(z) = \Delta_r(z)$. We will also denote
\[
D_T = V \cap D,
\]
which is a bounded symmetric domain of tube type. We will call $D_T$ the tube type domain associated to $D$. In fact, $D_T$ is the largest tube type domain obtained as the intersection of $D$ with a complex subspace of $\gp_+$.

On the other hand, there is a joint Peirce decomposition associated to $e_1, \dots, e_r$ given by
\[
\gp_+ = \bigoplus_{0 \leq j,k \leq r} V_{jk},
\]
where we define
\begin{align*}
V_{jj} &= V_1(e_j), \quad \text{for $j=1, \dots, r$}, \\
V_{jk} &= V_{\frac{1}{2}}(e_j) \cap V_{\frac{1}{2}}(e_k), \quad \text{for $j,k=1, \dots, r$}, \\
V_{j0} = V_{0j} &= V_{\frac{1}{2}}(e_j) \cap \bigcap_{k\not=j} V_0(e_k), \quad \text{for $j=1, \dots, r$}, \\
V_{00} &= V_0(e_1) \cap \dots \cap V_0(e_r).
\end{align*}
We note that since $e_1, \dots, e_r$ is maximal, it turns out that $V_{00} = 0$. It is also well known that there exists a pair of integers $a > 0$ and $b \geq 0$ such that
\begin{align*}
\dim V_{jj} &= 1, \quad \text{for $j=1, \dots, r$},  \\
\dim V_{jk} &= a, \quad \text{for $j,k=1, \dots, r$},  \\
\dim V_{j0} &= b, \quad \text{for $j=1, \dots, r$}.
\end{align*}
The numbers $a,b$ are the characteristic multiplicities of $D$. We observe that $D$ is of tube type if and only if $b = 0$. If we denote by $n$ and $n_T$ the dimensions of $D$ and $D_T$, respectively, then we have
\[
n = r + \frac{r(r-1)}{2}a + rb, \quad n_T = r + \frac{r(r-1)}{2}a.
\]
The genus of the domain $D$ is defined by
\[
p = \frac{n + n_T}{r} = 2 + (r-1)a + b.
\]

Finally, it is also well known that there is a set of strongly orthogonal roots $\gamma_1, \dots, \gamma_r$ in the sense of Harish-Chandra such that, for every $j = 1, \dots, r$, the tripotent $e_j$ belongs to the root space $\gg_{\gamma_j}$. Without loss of generality, we can assume that $\gamma_1 > \dots > \gamma_r > 0$.

\section{Analysis on bounded symmetric domains}
\label{sec:AnalysisBSD}
With the setup provided by Section~\ref{sec:JordanBSD} we have the following integral formula (see \cite{FarautKoranyi,Helgason2,UpmeierBook})
\begin{align}\label{eq:int_gpplus}
&\int_{\gp_+} f(z) \dif z \\
&= c \int_{[0,\infty)^r} \int_K f\bigg(g \cdot \sum_{j=1}^{r} t_j e_j\bigg) \dif g
\prod_{1\leq j<k \leq r}|t_j^2 - t_k^2|^a \prod_{j=1}^{r} t_j^{2b+1}  \dif t_1 \dots \dif t_r  \notag
\end{align}
where $\dif z$ is the Lebesgue measure on $\gp_+$ corresponding to its Hermitian structure, $\dif g$ is the probability Haar measure of $K$ and $c$ is a positive constant whose value will not be used.

We also have the following integral formula.
\begin{align}\label{eq:int_cone}
&\int_{\Omega} f(x) \dif x \\
&= c' \int_{[0,\infty)^r} \int_L f\Big(s\cdot \sum_{j=1}^{r}x_j e_j\Big) \dif s
\prod_{1\leq j < k \leq r} |x_j^2 - x_k^2|^a \dif x_1 \dots \dif x_r,  \notag
\end{align}
where $\dif x$ is the Lebesgue measure on $\Omega$, $\dif g$ is the Haar measure of $L$ and $c'$ is a constant whose value we will not need.

On the other hand, for the Lebesgue measure on $\gp_+$, we have (see \cite{KoranyiInvBSD})
\[
\vol(D) = \pi^n \prod_{j = 1}^{r}
\frac{\Gamma(p - \frac{n}{r} - (j-1) \frac{a}{2})}{\Gamma(p - (j-1)\frac{a}{2})}.
\]
Hence, we consider the normalized Lebesgue measure
\[
\dif v(z) = \frac{1}{\vol(D)} \dif z
= \frac{1}{\pi^n} \prod_{j = 1}^{r}
\frac{\Gamma(p - (j-1)\frac{a}{2})}{\Gamma(p - \frac{n}{r} - (j-1) \frac{a}{2})} \dif z.
\]
The (weigthless) Bergman space is the closed subspace $\mA^2(D)$ of $L^2(D, v)$ that consists of holomorphic functions. The reproducing kernel of this space can be described as follows.

Let us define the Bergman endomorphisms
\begin{align*}
B(x,y) : \gp_+ &\rightarrow \gp_+ \\
B(x,y)z &= z - 2\{xybz\} + \{x\{yzy\}x\},
\end{align*}
for any $x,y \in \gp_+$. Then, the reproducing $K_D$ of $\mA^2(D)$ is given by (see \cite{LoosBSDJordan})
\begin{align*}
K_D : D \times D &\rightarrow \C   \\
K_D(z,w) &= \det(B(z,w))^{-1}.
\end{align*}	

Furthermore, as noted in \cite{FarautKoranyi} there is a real polynomial $h(z,w)$ on $\gp_+ \times \gp_+$, holomorphic in $z$ and anti-holomorphic in $w$, such that
\[
h(z,w)^p = \det(B(z,w)).
\]
It can also be described as the unique $K$-invariant polynomial such that
\begin{equation}
\label{eq:h(x,x)}
h(x,x) = \Delta(e-x^2) = \prod_{j=1}^{r} (1 - x_j^2)
\end{equation}
for every $x = x_1 e_1 + \dots + x_r e_r$ ($x_j \in \R$) and with the square $x^2$ computed in the Jordan algebra $V$. Note that we have used equation~\eqref{eq:Deltaj} with $j = r$.

Hence, the Bergman kernel can also be computed as
\[
K_D(z,w) = h(z,w)^{-p},
\]
for every $z,w \in D$.

The following weighted volume computation can also be obtained (see \cite{UpmeierBook})
\[
\int_D h(z,z)^{\lambda-p} \dif z
= \pi^n \prod_{j=1}^{r}\frac{\Gamma(\lambda - \frac{n}{r} - (j-1) \frac{a}{2})}{\Gamma(\lambda - (j-1)\frac{a}{2})},
\]
which holds for every $\lambda > p-1$. Hence, for every such $\lambda$ we consider the weighted measure
\[
\dif v_\lambda(z) =
\frac{1}{\pi^n}
\prod_{j=1}^{r}\frac{\Gamma(\lambda - (j-1)\frac{a}{2})}{\Gamma(\lambda - \frac{n}{r} - (j-1) \frac{a}{2})}
h(z,z)^{\lambda-p} \dif z,
\]
which thus satisfies $v_\lambda(D) = 1$.

For every $\lambda > p-1$, the weighted Bergman space $\mA^2_\lambda(D)$ corresponding to $\lambda$ is given by
\[
\mA^2_\lambda(D) = \{ f \in L^2(D, v_\lambda) \mid f \text{ is holomorphic} \}.
\]
This is a closed subspace of $L^2(D, v_\lambda)$ and a reproducing kernel space with Bergman kernel given by
\begin{align*}
K_{D,\lambda} : D \times D &\rightarrow \C \\
K_{D,\lambda}(z,w) &= h(z,w)^{-\lambda}.
\end{align*}
In particular, the corresponding Bergman projection $B_{D,\lambda} : L^2(D,v_\lambda) \rightarrow \mA^2_\lambda(D)$ satisfies
\begin{align*}
B_{D,\lambda}(f)(z) &= \int_D f(w) h(z,w)^{-\lambda} \dif v_\lambda(w) \\
&= \frac{1}{\pi^n}
\prod_{j=1}^{r}\frac{\Gamma(\lambda - (j-1)\frac{a}{2})}{\Gamma(\lambda - \frac{n}{r} - (j-1) \frac{a}{2})}
\int_D \frac{f(w)h(w,w)^{\lambda-p} \dif w}{h(z,w)^\lambda}.
\end{align*}

We note that $\mA^2_p(D) = \mA^2(D)$ is the Bergman space corresponding to the (weightless) Lebesgue measure.

For any function $\varphi \in L^\infty(D)$, the Toeplitz operator on $\mA^2_\lambda(D)$ with bounded symbol $\varphi$ is defined by
\begin{align*}
T_\varphi : \mA^2_\lambda(D) &\rightarrow \mA^2_\lambda(D) \\
T_\varphi(f) &= B_{D,\lambda}(\varphi f).
\end{align*}

On the other hand, for every $\lambda > p-1$ there is a unitary representation of the universal covering group $\widetilde{G}$ of $G$
\begin{align*}
\pi_\lambda : \widetilde{G} \times \mA^2_\lambda(D) &\rightarrow \mA^2_\lambda(D) \\
(\pi_\lambda(g)f)(z) &= J(g^{-1},z)^{\frac{\lambda}{p}} f(g^{-1}z),
\end{align*}
where $J(g,z)$ denotes the complex Jacobian of the transformation $g$ at the point $z$. These yield the holomorphic discrete series of $G$.

Since $K$ is the subgroup of linear unitary transformations on $\gp_+$ that belong to $G$ we have $J(g,z) \equiv 1$ for all $g \in K$. In particular, for every $\lambda > p-1$ the above yields a unitary representation of $K$ itself given by
\begin{align*}
\pi_\lambda|_K : K \times \mA^2_\lambda(D) &\rightarrow \mA^2_\lambda(D)   \\
(\pi_\lambda(g)f)(z) &= f(g^{-1} z).
\end{align*}

\section{Radial operators and the isotypic decomposition of $\pi_\lambda|_K$}
\label{sec:iso-K_and_radial-operators}
Let $\pi : H \rightarrow \U(\mH)$ be a unitary representation of a Lie group $H$. We recall that a bounded operator $T : \mH \rightarrow \mH$ is called $\pi$-intertwining if we have
\[
T\circ \pi(h) = \pi(h)\circ T
\]
for every $h \in H$. The algebra of $\pi$-intertwining bounded operators is denoted by $\End_H(\mH)$. This algebra is a fundamental invariant to understand the decomposition of $\mH$ into irreducible subspaces. In particular, we have the following well known result where we have considered the case of compact groups.

\begin{proposition}
	\label{prop:mult-free}
	The unitary representation $\pi$ of the group $H$ on the Hilbert space $\mH$ is multiplicity-free if and only if $\End_H(\mH)$ is commutative. For $H$ compact, this occurs exactly when $\mH$ can be decomposed as a direct sum of inequivalent irreducible $H$-invariant closed subspaces.
\end{proposition}

By focusing our attention to the subgroup $K$ of $G$ and the holomorphic discrete series we obtain a special type of operators.

\begin{definition}
	\label{def:radial_operators}
	A bounded operator $T : \mA^2_\lambda(D) \rightarrow \mA^2_\lambda(D)$ is called radial if it is $\pi_\lambda|_K$-intertwining, i.e.~if $T$ satisfies
	\[
	T\circ \pi_\lambda(g) = \pi_\lambda(g)\circ T
	\]
	for every $g \in K$. The algebra of bounded radial operators on $\mA^2_\lambda(D)$ is denoted by $\End_K(\mA^2_\lambda(D))$.
\end{definition}

It is well known that the unitary representations $\pi_\lambda$ are irreducible highest weight representations of the group $\widetilde{G}$. The fact that $K$ is a maximal compact subgroup of $G$ thus implies that the restrictions $\pi_\lambda|_K$ are multiplicity-free: this follows, for example, from the results in \cite{Kobayashi} and \cite{OlafssonOrsted}.

As a consequence of the previous remarks we obtain the following.

\begin{proposition}
	\label{prop:En_K_commutative}
	For every $\lambda > p-1$, the algebra $\End_K(\mA^2_\lambda(D))$ is commutative.
\end{proposition}

We now consider the corresponding invariance property for symbols and their associated Toeplitz operators.

\begin{definition}
	A bounded symbol $\psi \in L^\infty(D)$ is called radial if it is $K$-invariant. In other words, if for every $g \in K$ we have
	\[
	\psi(g z) = \psi(z)
	\]
	for almost every $z \in D$.
\end{definition}

We recall the following well known result. It follows from \cite{Helgason} by using the fact that the real span of $e_1, \dots, e_r$ is a maximal Abelian subspace of $\gp_+$.

\begin{proposition}
	\label{prop:KA-decomposition_and_K-invsymbols}
	For every $z \in D$ there exist $g \in K$ and $t \in [0,1)^r$ such that
	\[
	z = g \sum_{j=1}^{r} t_j e_j.
	\]
	In particular, a bounded symbol $\psi \in L^\infty(D)$ is radial if and only if we have
	\[
	\psi(z) = \psi\Big(\sum_{j=1}^{r} t_j e_j\Big),
	\]
	when the above relation holds between $z \in D$ and $t \in [0,1)^r$.
\end{proposition}

It was observed in \cite{DOQJFA} the equivalence between the invariance of a symbol and of its associated Toeplitz operators. In particular, we have the following result which is a particular case of Corollary~3.3 and Theorem~5.2 from \cite{DOQJFA}.

\begin{proposition}[DQO \cite{DOQJFA}]
	\label{prop:intertwining_vs_invariance}
	For a given bounded symbol $\psi \in L^\infty(D)$ and a fixed $\lambda > p-1$ the following conditions are equivalent
	\begin{enumerate}
		\item The symbol $\psi$ is radial.
		\item The Toeplitz operator $T_\psi$ on $\mA^2_\lambda(D)$ is radial.
	\end{enumerate}
	Hence, if we denote by $\mA^R$ the space of bounded radial symbols and by $\mT^{(\lambda)}(\mA^R)$ the $C^*$-algebra generated by the Toeplitz operators on $\mA^2_\lambda(D)$ with bounded radial symbols, then we have
	\[
	\mT^{(\lambda)}(\mA^R) \subset \End_K(\mA^2_\lambda(D)).
	\]
	In particular, for every $\lambda > p-1$ the $C^*$-algebra $\mT^{(\lambda)}(\mA^R)$ is commutative.
\end{proposition}

We now describe the multiplicity-free decomposition of $\mA^2_\lambda(D)$ with respect to the representation $\pi_\lambda|_K$.

Recall that the space of polynomials $\mP(\gp_+)$ is a dense $\pi_\lambda|_K$-invariant subspace of $\mA^2_\lambda(D)$ for every $\lambda > p - 1$. Hence, it is enough to describe the isotypic decomposition of the representation of $K$ on $\mP(\gp_+)$ (see \cite{Schmid} and \cite{UpmeierBook}).

Let us denote by $\rN^r$ the space of $r$-tuples of integers $(\alpha_1, \dots, \alpha_r)$ satisfying the condition $\alpha_1 \geq \dots \geq \alpha_r \geq 0$. Then, the highest weights of the irreducible $K$-submodules contained in $\mP(\gp_+)$ are precisely those of the form
\begin{equation}\label{eq:highestweights}
-\alpha_1 \gamma_1 - \dots - \alpha_r \gamma_r
\end{equation}
for some $\alpha \in \rN^r$. Furthermore, for each $\alpha \in \rN^r$ there is exactly one such $K$-submodule that we will denote by $\mP^\alpha(\gp_+)$. In particular, we have a $K$-invariant algebraic direct sum
\[
\mP(\gp_+) = \bigoplus_{\alpha \in \rN^r} \mP^\alpha(\gp_+),
\]
that induces the multiplicity-free isotypic decomposition for the representation $\pi_\lambda|_K$ given by
\[
\mA^2_\lambda(D) = \bigoplus_{\alpha \in \rN^r} \mP^\alpha(\gp_+).
\]
In particular, for every weight of the form \eqref{eq:highestweights} the corresponding space of highest vectors in $\mP(\gp_+)$ is $1$-dimensional.

The following is an immediate consequence of the previous discussion.

\begin{proposition}
	\label{prop:Schmid_diagonal}
	For every radial operator $T : \mA^2_\lambda(D) \rightarrow \mA^2_\lambda(D)$ and for every $\alpha \in \rN^r$ let us denote by $c_\alpha(T)$ the complex number such that
	\[
	T|_{\mP^\alpha(\gp_+)} = c_\alpha(T) I_{\mP^\alpha(\gp_+)}.
	\]
	Then, the map defined by
	\begin{align*}
	\End_{\pi_\lambda|_K}(\mA^2_\lambda(D)) &\rightarrow \ell_\infty(\rN^r)  \\
	T &\mapsto (c_\alpha(T))_{\alpha \in \rN^r},
	\end{align*}
	is an isomorphism of $C^*$-algebras. Furthermore, for every radial operator $T$ we have
	\[
	c_\alpha(T) = \frac{\left<T\phi,\phi\right>_\lambda}{\left<\phi,\phi\right>_\lambda},
	\]
	for every non-zero $\phi \in \mP^\alpha(\gp_+)$ and for every $\alpha \in \rN^r$.
\end{proposition}

On the other hand (see \cite{Johnson,FarautKoranyi,UpmeierBook}), it is well known that the polynomial $\Delta_j(z)$ on $\gp_+$ is a highest weight vector for $\pi_\lambda|_K$ with corresponding weight $-\gamma_1 - \dots - \gamma_j$, for every $j = 1, \dots, r$. In particular, the space of highest weight vectors is the free polynomial algebra $\C[\Delta_1(z), \dots, \Delta_r(z)]$. As a consequence of this, a highest weight vector corresponding to \eqref{eq:highestweights} is given by
\[
\Delta_\alpha(z) = \Delta_1(z)^{\alpha_1 - \alpha_2} \Delta_2(z)^{\alpha_2 - \alpha_3} \dots \Delta_r(z)^{\alpha_r},
\]
for every $\alpha \in \rN^r$. These are called the conical polynomials associated to $\pi_\lambda|_K$.

We have the following naturally defined polynomial
\[
\phi_\alpha(z) = \int_L \Delta_\alpha(sz) \dif s,
\]
for every $\alpha \in \rN^r$, which belongs to $\mP^\alpha(\gp_+)$ and that is ($L$-)spherical for the representation $\pi_\lambda|_K$ restricted to $\mP^\alpha(\gp_+)$.

The next result is proved, for example, in \cite{UpmeierBook} (see Theorem~2.8.10 therein).

\begin{lemma}
	\label{lem:K-int_z2}
	With the above notation we have
	\[
	\int_K |\phi_\alpha(k z)|^2 \dif k = \frac{1}{d_\alpha}\phi_\alpha(z^2),
	\]
	for every $\alpha \in \rN^r$ and $z \in V$, where $d_\alpha = \dim \mP^\alpha(\gp_+)$ and for $z^2$ computed using the Jordan algebra structure of $V$ described in Section~\ref{sec:JordanBSD}.
\end{lemma}

The previous results allow us to compute the coefficients described in Proposition~\ref{prop:Schmid_diagonal} for Toeplitz operators. The first step is given by the following result. We recall that for every $x$ in the cone $\Omega$ there exists a unique $y \in \Omega$ such that $y^2 = x$, and we denote $y = \sqrt{x}$.

\begin{theorem}
	\label{thm:RadialToepNormComputation}
	Let $D$ be a bounded symmetric domain and $\psi \in L^\infty(D)$ a radial symbol. Then, for every $\lambda > p-1$ and for every $\alpha \in \rN^r$ we have
	\begin{align*}
	\left<T_\psi \phi_\alpha, \phi_\alpha\right>_\lambda &=
	\left<\psi \phi_\alpha, \phi_\alpha\right>_\lambda \\
	&= C \int_{\Omega\cap(e-\Omega)}
	\psi(\sqrt{x}) \Delta_\alpha(x) \Delta(e-x)^{\lambda-p} \Delta(x)^b \dif x,
	\end{align*}
	for some constant $C$ independent of $\psi$ and $\alpha$.
\end{theorem}
\begin{proof}
	We fix $\psi$, $\lambda$ and $\alpha$ to compute as follows.
	
	First, we apply the integral formula \eqref{eq:int_gpplus} to obtain
	\begin{multline*}
	\left<\psi \phi_\alpha, \phi_\alpha\right>_\lambda
	= c_\lambda \int_D \psi(z)|\phi_\alpha(z)|^2 h(z,z)^{\lambda-p} \dif z  \\
	= c c_\lambda \int_{[0,1)^r} \int_K \psi\Big(g\cdot \sum_{j=1}^{r} t_j e_j\Big)
	\Big|\phi_\alpha\Big(g\cdot \sum_{j=1}^{r} t_j e_j\Big)\Big|^2
	h\Big(g\cdot \sum_{j=1}^{r} t_j e_j,g\cdot \sum_{j=1}^{r} t_j e_j\Big)^{\lambda-p}
	\dif g  \\
	\times
	\prod_{1\leq j<k \leq r}|t_j^2 - t_k^2|^a \prod_{j=1}^{r} t_j^{2b+1}  \dif t_1 \dots \dif t_r,
	\end{multline*}
	where we take
	\[
	c_\lambda = \frac{1}{\pi^n}
	\prod_{j=1}^{r}\frac{\Gamma(\lambda - (j-1)\frac{a}{2})}{\Gamma(\lambda - \frac{n}{r} - (j-1) \frac{a}{2})}
	\]
	Note that we have restricted the integral to $[0,1)^r$ in order to integrate over $D$. Since $\psi,h$ are $K$-invariant and by Lemma~\ref{lem:K-int_z2} we now obtain
	\begin{align*}
	\left<\psi \phi_\alpha, \phi_\alpha\right>_\lambda &= \\
	= & \frac{cc_\lambda}{d_\alpha} \int_{[0,1)^r} \psi\Big(\sum_{j=1}^{r} t_j e_j\Big)
	\phi_\alpha\Big(\sum_{j=1}^{r} t_j^2 e_j\Big)
	h\Big(\sum_{j=1}^{r} t_j e_j, \sum_{j=1}^{r} t_j e_j\Big)^{\lambda-p}  \\
	& \times
	\prod_{1\leq j<k \leq r}|t_j^2 - t_k^2|^a \prod_{j=1}^{r} t_j^{2b+1}  \dif t_1 \dots \dif t_r \\
	\intertext{and using equation~\eqref{eq:h(x,x)} we get}
	= & \frac{cc_\lambda}{d_\alpha} \int_{[0,1)^r} \psi\Big(\sum_{j=1}^{r} t_j e_j\Big)
	\phi_\alpha\Big(\sum_{j=1}^{r} t_j^2 e_j\Big)
	\Delta\Big(e - \sum_{j=1}^{r}t_j^2 e_j\Big)^{\lambda-p}  \\
	& \times
	\prod_{1\leq j<k \leq r}|t_j^2 - t_k^2|^a \prod_{j=1}^{r} t_j^{2b+1}  \dif t_1 \dots \dif t_r \\	
	= & \frac{cc_\lambda}{d_\alpha} \int_{[0,1)^r} \psi\Big(\sum_{j=1}^{r} \sqrt{x_j} e_j\Big)
	\phi_\alpha\Big(\sum_{j=1}^{r} x_j e_j\Big)
	\Delta\Big(e - \sum_{j=1}^{r}x_j e_j\Big)^{\lambda-p}   \\
	& \times
	\prod_{1\leq j<k \leq r}|x_j - x_k|^a \prod_{j=1}^{r} x_j^b  \dif x_1 \dots \dif x_r \\	
	\intertext{using equation~\eqref{eq:Deltaj} we get}
	= & \frac{cc_\lambda}{d_\alpha} \int_{[0,1)^r} \psi\Big(\sum_{j=1}^{r} \sqrt{x_j} e_j\Big)
	\phi_\alpha\Big(\sum_{j=1}^{r} x_j e_j\Big)	
	\Delta\Big(e - \sum_{j=1}^{r}x_j e_j\Big)^{\lambda-p} \\
	& \times
	\Delta\Big(\sum_{j=1}^{r}x_j e_j\Big)^b
	\prod_{1\leq j<k \leq r}|x_j - x_k|^a  \dif x_1 \dots \dif x_r \\		
	\intertext{and then applying the definition of $\varphi_\alpha$ we obtain}
	= & \frac{cc_\lambda}{d_\alpha} \int_{[0,1)^r} \psi\Big(\sum_{j=1}^{r} \sqrt{x_j} e_j\Big)
	\Delta\Big(e - \sum_{j=1}^{r}x_j e_j\Big)^{\lambda-p}
	\Delta\Big(\sum_{j=1}^{r}x_j e_j\Big)^b  \\
	& \times
	\int_L \Delta_\alpha\Big(s\cdot \sum_{j=1}^{r} x_j e_j\Big) \dif s
	\prod_{1\leq j<k \leq r}|x_j - x_k|^a  \dif x_1 \dots \dif x_r. \\	
	\end{align*}
	
	Since $\psi, \Delta$ and $e$ are $L$-invariant, we are now in position to apply equation~\eqref{eq:int_cone}. To do so, we note that integration over $[0,1)^r$ for
	\[
	x = \sum_{J=1}^{r} x_j e_j,
	\]
	corresponds to the set $\Omega \cap (e-\Omega)$. Hence, we obtain
	\[
	\left<\psi \phi_\alpha, \phi_\alpha\right>_\lambda
	= C \int_{\Omega\cap(e-\Omega)}
	\psi(\sqrt{x}) \Delta_\alpha(x) \Delta(e-x)^{\lambda-p} \Delta(x)^b \dif x,
	\]
	where $C$ is a constant that is independent of $\psi$ and $\alpha$. Note that we have used
	\[
	\sqrt{\sum_{j=1}^{r} x_j e_j} = \sum_{j=1}^{r} \sqrt{x_j} e_j,
	\]
	which follows from the mutual orthogonality of $e_1, \dots, e_r$.
\end{proof}

As a consequence of the proof of Theorem~\ref{thm:RadialToepNormComputation} we have the following.

\begin{corollary}
	\label{cor:RadialToepNormComputationCoord}
	With the assumptions of Theorem~\ref{thm:RadialToepNormComputation} we have for every $\lambda > p-1$ and $\alpha \in \rN^r$
	\begin{multline*}
	\left<T_\psi \phi_\alpha, \phi_\alpha\right>_\lambda =
	\left<\psi \phi_\alpha, \phi_\alpha\right>_\lambda \\
	= C \int_{[0,1)^r} \psi\Big(\sum_{j=1}^{r} \sqrt{x_j} e_j\Big)
	\Delta_\alpha\Big(\sum_{j=1}^{r} x_j e_j\Big)
	\Delta\Big(\sum_{j=1}^{r}(1-x_j) e_j\Big)^{\lambda-p}
	\Delta\Big(\sum_{j=1}^{r}x_j e_j\Big)^b  \\
	\times
	\prod_{1\leq j<k \leq r}|x_j - x_k|^a  \dif x_1 \dots \dif x_r,
	\end{multline*}
	for some constant $C$ independent of $\psi$ and $\alpha$.
\end{corollary}

A direct application of Theorem~\ref{thm:RadialToepNormComputation} yields the following expression of the coefficients in Proposition~\ref{prop:Schmid_diagonal} for radial Toeplitz operators in terms of the polynomials $\Delta_\alpha(z)$ and $\Delta_j(z)$. We recall that for our previous notation we have $\Delta(z) = \Delta_r(z)$ for a domain of rank $r$.

\begin{theorem}
	\label{thm:Schmid_diagonal}
	Let $D$ be a bounded symmetric domain and $\psi \in L^\infty(D)$ a radial symbol. Then, for every $\lambda > p-1$ and for every $\alpha \in \rN^r$ the coefficients $c_\alpha(T_\psi)$ from Proposition~\ref{prop:Schmid_diagonal} satisfy
	\begin{align*}
	c_\alpha(T_\psi)
	&= \frac{\left<T_\psi \phi_\alpha,\phi_\alpha\right>_\lambda}
	{\left<\phi_\alpha,\phi_\alpha\right>_\lambda}
	= \frac{\displaystyle\int_{\Omega\cap(e-\Omega)}
		\psi(\sqrt{x}) \Delta_\alpha(x) \Delta(e-x)^{\lambda-p} \Delta(x)^b \dif x}
	{\displaystyle\int_{\Omega\cap(e-\Omega)}
		\Delta_\alpha(x) \Delta(e-x)^{\lambda-p} \Delta(x)^b \dif x}   \\
	&= \frac{\displaystyle\int_{\Omega\cap(e-\Omega)}
		\psi(\sqrt{x})
		\Delta_1(x)^{\alpha_1 - \alpha_2} \Delta_2(x)^{\alpha_2 - \alpha_3} \dots
		\Delta_r(x)^{\alpha_r+b} \Delta_r(e-x)^{\lambda-p} \dif x}
	{\displaystyle\int_{\Omega\cap(e-\Omega)}
		\Delta_1(x)^{\alpha_1 - \alpha_2} \Delta_2(x)^{\alpha_2 - \alpha_3} \dots
		\Delta_r(x)^{\alpha_r+b} \Delta_r(e-x)^{\lambda-p} \dif x}.	
	\end{align*}
\end{theorem}

We can now apply Corollary~\ref{cor:RadialToepNormComputationCoord} and use \eqref{eq:Deltaj} to write down the formulas from Theorem~\ref{thm:Schmid_diagonal} in terms of Lebesgue integrals over $[0,1)^r$ to obtain the following result.

\begin{theorem}
	\label{thm:Schmid_diagonal_Lebesgue}
	Let $D$ be a bounded symmetric domain and $\psi \in L^\infty(D)$ a radial symbol. Then, for every $\lambda > p-1$ and for every $\alpha \in \rN^r$ the coefficients $c_\alpha(T_\psi)$ from Proposition~\ref{prop:Schmid_diagonal} satisfy
	\begin{multline*}
	c_\alpha(T_\psi)
	= \frac{\left<T_\psi \phi_\alpha,\phi_\alpha\right>_\lambda}
	{\left<\phi_\alpha,\phi_\alpha\right>_\lambda}  = \\
	\frac{\displaystyle\int_{[0,1)^r}
		\psi\Big(\sum_{j=1}^{r} \sqrt{x_j} e_j\Big)
		\prod_{j=1}^{r} x_j^{\alpha_j+b} \prod_{j=1}^{r}(1-x_j)^{\lambda-p}
		\prod_{1\leq j<k \leq r}|x_j - x_k|^a \dif x_1 \dots \dif x_r}
	{\displaystyle\int_{[0,1)^r}
		\prod_{j=1}^{r} x_j^{\alpha_j+b} \prod_{j=1}^{r}(1-x_j)^{\lambda-p}
		\prod_{1\leq j<k \leq r}|x_j - x_k|^a \dif x_1 \dots \dif x_r}.	
	\end{multline*}
\end{theorem}

\section{Radial Toeplitz operators on the Bergman spaces}
\label{sec:radial_Toeplitz}
All the data associated to the domain $D$ involved in Theorem~\ref{thm:Schmid_diagonal_Lebesgue} is readily available in the references (see \cite{Helgason,UpmeierBook}). We collect this information and provide specific formulas for all irreducible bounded symmetric domains. Note that all the information for the classical cases are explicitly contained in the references, but not so in the exceptional cases. For the latter we provide the required arguments.

\subsection{Type I} Up to finite index covering maps the domain $D_{m,n}^I$ is given by the groups
\[
G = \SU(m,n), \quad K = \Spe(\U(m) \times \U(n)).
\]
For simplicity and without loss of generality we will assume that $m \leq n$. The elements of $K$ are matrices of the form
\[
\begin{pmatrix}
A & 0 \\
0 & B
\end{pmatrix},
\]
where $A \in \U(m), B \in \U(n)$ and $\det(A)\det(B) = 1$.

The domain in this case is given by
\[
D_{m,n}^I = \{ Z \in M_{m \times n}(\C) \mid ZZ^* < I_m\},
\]
where the $K$-action has the expression
\[
\begin{pmatrix}
A & 0 \\
0 & B
\end{pmatrix} Z = AZB^{-1},
\]
where $A,B$ are as above and $Z \in D^I_{m,n}$. And we have
\[
\dim D^I_{m,n} = mn, \quad r = m, \quad a = 2, \quad b = n-m, \quad p = n+m.
\]

A maximal collection of mutually orthogonal primitives is given by the $m \times n$ matrices $E_1, \dots, E_m$, where $E_j$ has $1$ at the $(j,j)$-th entry and $0$ elsewhere. Hence, we have
\[
E = (I_m, 0),
\]
where $0$ denotes here a $m \times (n-m)$ matrix. Proposition~\ref{prop:KA-decomposition_and_K-invsymbols} states that for every $Z \in D_{m,n}^I$ there exist $A \in \U(m), B \in \U(n)$ and a diagonal matrix $D$ of size $m \times m$ with entries in $[0,1)$ such that
\[
Z = A (D,0) B^{-1}.
\]

It follows that for every $\lambda > n+m-1$, $\psi$ bounded radial symbol and $\alpha \in \rN^m$ we have
\begin{multline*}
c_\alpha(T_\psi)
= \frac{\left<T_\psi \phi_\alpha,\phi_\alpha\right>_\lambda}
{\left<\phi_\alpha,\phi_\alpha\right>_\lambda} = \\
\frac{\displaystyle\int_{[0,1)^m}
	\psi(D(\sqrt{x}),0)
	\prod_{j=1}^{m} x_j^{\alpha_j+n-m} \prod_{j=1}^{m}(1-x_j)^{\lambda-m-n}
	\prod_{1\leq j<k \leq m}|x_j - x_k|^2 \dif x_1 \dots \dif x_m}
{\displaystyle\int_{[0,1)^m}
	\prod_{j=1}^{m} x_j^{\alpha_j+n-m} \prod_{j=1}^{m}(1-x_j)^{\lambda-m-n}
	\prod_{1\leq j<k \leq m}|x_j - x_k|^2 \dif x_1 \dots \dif x_m},	
\end{multline*}
where $D(\sqrt{x})$ denotes the $m\times m$ diagonal matrix whose entries along the diagonal are $\sqrt{x_1}, \dots, \sqrt{x_m}$.

If we assume that $m = 1$, then the domain is the unit ball $\mathbb{B}^n$ in $\C^n$ and the previous formula reduces to Theorem~6.2 from \cite{QuirogaGrudsky60}.

\subsection{Type II} Up to finite index covering maps the domain $D^{II}_m$ is given by the groups
\begin{align*}
G &= \SO^*(2m) = \left\{
\begin{pmatrix*}[r]
A & B \\
-\overline{B} & \overline{A}
\end{pmatrix*} \Big| A,B \in M_m(\C)
\right\} \cap \SU(m,m) \\
K &= \U(n) \cong \left\{
\begin{pmatrix*}[r]
A & 0 \\
0 & \overline{A}
\end{pmatrix*} \Big| A \in \U(m)
\right\}.
\end{align*}

The domain is give by
\[
D^{II}_m = \{ Z \in M_m(\C) \mid Z Z^* < I_m, Z^t = -Z \},
\]
with $K$-action
\[
A\cdot Z = A Z A^t,
\]
where $A \in \U(n)$ and $Z \in D^{II}_m$. The properties of this domain depend on the parity of $m$. Hence, we will write $m = 2n + \epsilon$ where $\epsilon \in \{0,1\}$. With this notation we now have
\[
\dim D^{II}_m = n(2n+2\epsilon-1), \quad r = n, \quad a = 4, \quad b = 2\epsilon, \quad
p = 4n + 2\epsilon - 2.
\]

For every $j = 1, \dots, n$ let us denote by $E_j \in M_m(\C)$ the block diagonal matrix of the form
\[
\begin{pmatrix}
0 & \cdots & 0 & \cdots  \\
\vdots & \ddots & \vdots \\
0 & \cdots & \begin{matrix*}[r]
0 & 1 \\
-1 & 0
\end{matrix*} & \cdots \\
\vdots & & \vdots &\ddots
\end{pmatrix},
\]
where the only non-zero block is the $(j,j)$-th one. In this description the blocks are $2 \times 2$ from the upper left corner and the $j$-th diagonal block has the pictured entry. Note that for $\epsilon = 1$, we have zero blocks of size $2 \times 1$ at the far right, of size $1 \times 2$ at the lower bottom and a single $0$ at the lower right corner. Hence, $E_1, \dots, E_n$ is a maximal collection of mutually orthogonal primitives. From this it is easy to write down $E$.

For this case, Proposition~\ref{prop:KA-decomposition_and_K-invsymbols} states that for every $Z \in D^{II}_m$ there exist $A \in \U(n)$, $t_1, \dots, t_n \in [0,1)$ such that
\[
Z = A D(t) A^t,
\]
where $D(t) = D(t_1, \dots, t_n) = \mathrm{Diag}(d(t_1), \dots, d(t_j), \dots)$ is a $m \times m$ matrix that is block diagonal with $2 \times 2$ blocks along the diagonal starting from the upper left corner and given by
\[
d(t_j) = \begin{pmatrix*}[r]
0 & t_j \\
-t_j & 0
\end{pmatrix*}.
\]
As before, we have to complete the matrix $D(t)$ with some smaller zero blocks for the case $\epsilon = 1$.

It follows that for $\lambda > 4n + 2\epsilon - 3$, $\psi$ bounded radial symbol and $\alpha \in \rN^n$ we have
\begin{multline*}
c_\alpha(T_\psi)
= \frac{\left<T_\psi \phi_\alpha,\phi_\alpha\right>_\lambda}
{\left<\phi_\alpha,\phi_\alpha\right>_\lambda} = \\
\frac{\displaystyle\int_{[0,1)^n}
	\psi(D(\sqrt{x}))
	\prod_{j=1}^{n} x_j^{\alpha_j+2\epsilon} \prod_{j=1}^{n}(1-x_j)^{\lambda-4n-2\epsilon+2}
	\prod_{1\leq j<k \leq n}|x_j - x_k|^4 \dif x_1 \dots \dif x_n}
{\displaystyle\int_{[0,1)^n}
	\prod_{j=1}^{n} x_j^{\alpha_j+2\epsilon} \prod_{j=1}^{n}(1-x_j)^{\lambda-4n-2\epsilon+2}
	\prod_{1\leq j<k \leq n}|x_j - x_k|^4 \dif x_1 \dots \dif x_n},
\end{multline*}
where the block diagonal matrix $D(\sqrt{x}) = D(\sqrt{x_1}, \dots, \sqrt{x_n})$ is given as before.

\subsection{Type III} Up to finite index covering maps the domain $D^{III}_n$ is given by the groups
\begin{align*}
G &= \Symp(n,\R), \\
K &= \U(n) = \left\{\begin{pmatrix}
A & 0 \\
0 & \overline{A}
\end{pmatrix} \Big| A \in \U(n) \right\}.		
\end{align*}

The domain is naturally identified is given by
\[
D^{III}_n = \{Z \in M_n(\C) \mid Z Z^* < I_n, Z^t = Z\},
\]
with $K$-action
\[
A\cdot Z = AZA^t,
\]
where $A \in \U(n)$ and $Z \in D^{III}_n$. For this domain we have
\[
\dim D^{III}_n = \frac{n(n+1)}{2}, \quad r = n, \quad a = 1, \quad b = 0, \quad p = n+1.
\]

A maximal collection of mutually orthogonal primitives is given by $E_1, \dots, E_n$, the canonical basis for the space of diagonal matrices in $M_n(\C)$. In particular, we have $E = I_n$. Proposition~\ref{prop:KA-decomposition_and_K-invsymbols} states that for every $Z \in D^{III}_n$ there exist $A \in \U(n)$ and a diagonal matrix $D$ with entries in $[0,1)$ such that
\[
Z = A D A^t.
\]

It follows that for every $\lambda > n$, $\psi$ a bounded radial symbol and $\alpha \in \rN^n$ we have
\begin{multline*}
c_\alpha(T_\psi)
= \frac{\left<T_\psi \phi_\alpha,\phi_\alpha\right>_\lambda}
{\left<\phi_\alpha,\phi_\alpha\right>_\lambda} = \\
\frac{\displaystyle\int_{[0,1)^n}
	\psi(D(\sqrt{x}))
	\prod_{j=1}^{n} x_j^{\alpha_j} \prod_{j=1}^{n}(1-x_j)^{\lambda-n-1}
	\prod_{1\leq j<k \leq r}|x_j - x_k| \dif x_1 \dots \dif x_n}
{\displaystyle\int_{[0,1)^n}
	\prod_{j=1}^{n} x_j^{\alpha_j} \prod_{j=1}^{n}(1-x_j)^{\lambda-n-1}
	\prod_{1\leq j<k \leq m}|x_j - x_k| \dif x_1 \dots \dif x_m},	
\end{multline*}
where $D(\sqrt{x})$ denotes the $n\times n$ diagonal matrix whose entries along the diagonal are $\sqrt{x_1}, \dots, \sqrt{x_n}$.

\subsection{Type IV} The groups associated to the domain $D^{IV}_n$ can be taken to be
\begin{align*}
G &= \SO(n,2) \\
K &= \SO(n) \times \SO(2).
\end{align*}

For this domain we have the realization
\[
D^{IV}_n = \Big\{ z \in \C^n \Big| |z|^2 < 2,\;  |z|^2 < 1 + \Big|\frac{1}{2} \sum_{j=1}^{n} z_j^2\Big| \Big\}
\]
with the $K$-action
\[
\begin{pmatrix*}[r]
A & 0 \\
0 & B
\end{pmatrix*} \cdot z = e^{-i\theta}A z,
\]
where $z \in D^{IV}_n$, $A \in \SO(n)$ and $B \in SO(2)$ has the form
\[
B =
\begin{pmatrix*}[r]
\cos(\theta) & -\sin(\theta) \\
\sin(\theta) & \cos(\theta)
\end{pmatrix*}.
\]
This domain has the following data
\[
\dim D^{IV}_n = n, \quad r = 2, \quad a = n-1, \quad b = 0, \quad p = n.
\]

A maximal collection of mutually orthogonal primitives is given by
\begin{align*}
e_1 &= \Big(\frac{1}{2}, \frac{i}{2}, 0, \dots, 0\Big) \\
e_2 &= \Big(\frac{1}{2}, -\frac{i}{2}, 0, \dots, 0\Big),
\end{align*}
and so we have $e = (1, 0, \dots, 0)$. Proposition~\ref{prop:KA-decomposition_and_K-invsymbols} says that for every $z \in D^{IV}_n$ there is $A \in \SO(n)$, $\theta \in \R$ and $t_1, t_2 \in [0,1)$ such that
\[
e^{i\theta}A z = t_1 e_1 + t_2 e_2.
\]

We now have that for every $\lambda > n-1$, $\psi$ a bounded radial symbol and $\alpha \in \rN^2$ we have
\begin{multline*}
c_\alpha(T_\psi)
= \frac{\left<T_\psi \phi_\alpha,\phi_\alpha\right>_\lambda}
{\left<\phi_\alpha,\phi_\alpha\right>_\lambda} = \\
\frac{\displaystyle\int_{[0,1)^2}
	\psi(\sqrt{x_1} e_1 + \sqrt{x_2} e_2)
	x_1^{\alpha_1} x_2^{\alpha_2} (1-x_1)^{\lambda-n} (1-x_2)^{\lambda-n}
	|x_1 - x_2|^{n-1} \dif x_1 \dif x_2}
{\displaystyle\int_{[0,1)^2}
	x_1^{\alpha_1} x_2^{\alpha_2} (1-x_1)^{\lambda-n} (1-x_2)^{\lambda-n}
	|x_1 - x_2|^{n-1} \dif x_1 \dif x_2}.	
\end{multline*}

\subsection{Exceptional type V} In the notation of \cite{Helgason}, the symmetric pair of Lie algebras that realize this domain is $(\gge_{6(-14)}, \so(10) + \R)$, also known as $E\,III$.

The properties of this domain are better described in terms of the octonians $\Oct$. We recall that $\Oct$ is an $8$-dimensional real composition algebra whose basic properties can be found in the literature (see \cite{SpringerVeldkamp}). We consider the complexified octonions $\Oct_\C = \Oct \otimes_\R \C$ and the $16$-dimensional complex vector space $\Oct_\C^2$. According to \cite{LoosBSDJordan}, the latter carries a Jordan pair structure that such that
\[
\{xyx\} = x(\widetilde{y}^tx),
\]
for every $x, y \in \Oct_\C^2$. Here, $\widetilde{y}$ is computed component-wise and comes from the the involution on $\Oct_\C$ obtained by complexifying the canonical conjugation on $\Oct$.

The exceptional domain $D^V$ has $\Oct_\C^2$ as associated Jordan pair (see \cite{LoosBSDJordan}) and so it is a bounded domain in this complex vector space. The data for this domain is the following
\[
\dim D^V = 16, \quad r = 2, \quad a = 6, \quad b = 4, \quad p = 12.
\]
Furthermore, using the Jordan pair structure described it is easy to check that a maximal collection of mutually orthogonal primitives is given by
\[
e_1 = (1,0), \quad e_2 = (0,1),
\]
and so $e = (1,1)$. In this case, Proposition~\ref{prop:KA-decomposition_and_K-invsymbols} states that for every $(a,b) \in D^V$ there exists $g \in K$ (where $K \simeq \SO(10) \times \T$) and $(t_1, t_2) \in [0,1)^2$ such that
\[
(a,b) = g(t_1, t_2).
\]

With respect to Toeplitz operators we now conclude that for every $\lambda > 11$, $\psi$ a bounded radial symbol and $\alpha \in \rN^2$ we have
\begin{multline*}
c_\alpha(T_\psi)
= \frac{\left<T_\psi \phi_\alpha,\phi_\alpha\right>_\lambda}
{\left<\phi_\alpha,\phi_\alpha\right>_\lambda} = \\
\frac{\displaystyle\int_{[0,1)^2}
	\psi(\sqrt{x_1},\sqrt{x_2})
	x_1^{\alpha_1+4}2_1^{\alpha_2+4} (1-x_1)^{\lambda-12} (1-x_2)^{\lambda-12}
	|x_1 - x_2|^6 \dif x_1 \dif x_2}
{\displaystyle\int_{[0,1)^2}
	x_1^{\alpha_1+4}2_1^{\alpha_2+4} (1-x_1)^{\lambda-12} (1-x_2)^{\lambda-12}
	|x_1 - x_2|^6 \dif x_1 \dif x_2}.	
\end{multline*}

\subsection{Exceptional type VI} In the notation of \cite{Helgason}, the symmetric pair of Lie algebras that realize this domain is $(\gge_{7(-25)}, \gge_6 + \R)$ which is also known as $E\,VII$.

Let us consider the $27$-dimensional complex vector space $H_3(\Oct_\C)$ of Hermitian $3 \times 3$ matrices with entries in $\Oct_\C$. By Hermitian we mean to satisfy
\[
\widetilde{x}^t = x.
\]
In particular, the elements of $H_3(\Oct_\C)$ are of the form
\[
\begin{pmatrix}
z_1 & a & b \\
\widetilde{a} & z_2 & c \\
\widetilde{b} & \widetilde{c} & z_3
\end{pmatrix},
\]
where $z_1, z_2, z_3 \in \C$ and $a,b,c \in \Oct_\C$. For $x,y \in H_3(\Oct_\C)$ denote the commutative product
\[
x \circ y = xy+yx.
\]
Then, $H_3(\Oct_\C)$ is a Jordan pair for the operation
\[
\{xyx\} = \frac{1}{2}(x \circ (x \circ y) - x^2 \circ y)
\]
for every $x, y \in H_3(\Oct_\C)$, where $x^2$ is just the square with the respect to product of matrices. As shown in \cite{LoosBSDJordan}, this is the Jordan pair associated to the domain $D^{VI}$. In particular, $D^{VI}$ has a realization as a circled bounded domain in $H_3(\Oct_\C)$. The data for this domain is given by
\[
\dim D^{VI} = 27, \quad r = 3, \quad a = 8, \quad b = 0, \quad p = 26.
\]
From the previous Jordan pair structure it is easy to see that a maximal collection of mutually orthogonal primitives is
\[
e_1 = E_{1,2} + E_{2,1}, \quad e_2 = E_{1,3} + E_{3,1}, \quad e_3 = E_{2,3} + E_{3,2},
\]
where as usual $E_{j,k}$ denotes the $3 \times 3$ matrix with all entries $0$ except at the position $(j,k)$ where it is $1$. For this case, Proposition~\ref{prop:KA-decomposition_and_K-invsymbols} implies that for every $x \in D^{VI}$ there are $g \in K$ (where $K \simeq E_6 \times \T$) and $(t_1, t_2, t_3) \in [0,1)^3$ such that
\[
x = g (t_1 e_1 + t_2 e_2 + t_3 e_3) =
g
\begin{pmatrix}
0 & t_1 & t_2 \\
t_1 & 0 & t_3 \\
t_2 & t_3 & 0
\end{pmatrix}.
\]

And for the Toeplitz operators we have that for every $\lambda > 25$, $\psi$ a bounded radial symbol and $\alpha \in \rN^3$ we have
\begin{multline*}
c_\alpha(T_\psi)
= \frac{\left<T_\psi \phi_\alpha,\phi_\alpha\right>_\lambda}
{\left<\phi_\alpha,\phi_\alpha\right>_\lambda} = \\
\frac{\displaystyle\int_{[0,1)^3}
	\psi\Big(\sum_{j=1}^{3} \sqrt{x_j} e_j\Big)
	\prod_{j=1}^{3} x_j^{\alpha_j} \prod_{j=1}^{3}(1-x_j)^{\lambda-26}
	\prod_{1\leq j<k \leq 3}|x_j - x_k|^8 \dif x_1 \dif x_2 \dif x_3}
{\displaystyle\int_{[0,1)^3}
	\prod_{j=1}^{3} x_j^{\alpha_j} \prod_{j=1}^{3}(1-x_j)^{\lambda-26}
	\prod_{1\leq j<k \leq 3}|x_j - x_k|^8 \dif x_1 \dif x_2 \dif x_r}.	
\end{multline*}

\bibliographystyle{amsplain}

\end{document}